\documentclass[12pt]{amsart}

\usepackage{amssymb}

\newtheorem{theorem}{Theorem}[section]
\newtheorem{corollary}[theorem]{Corollary}
\newtheorem{lemma}[theorem]{Lemma}
\newtheorem{proposition}[theorem]{Proposition}

\theoremstyle{definition}
\newtheorem{definition}[theorem]{Definition}

\theoremstyle{remark}
\newtheorem{remark}[theorem]{Remark}

\numberwithin{equation}{section}

\newcommand{\dbar}  {\overline{\partial}}
\newcommand{\del}   {\partial}
\newcommand{\eps}   {\varepsilon}
\newcommand{\tensor}{\otimes}

\newcommand{\CC}{\mathbb{C}}
\newcommand{\RR}{\mathbb{R}}

\newcommand{\A}{\mathcal{A}}

\DeclareMathOperator{\Ad}{Ad}
\DeclareMathOperator{\End}{End}
\DeclareMathOperator{\Herm}{Herm}
\DeclareMathOperator{\Hom}{Hom}
\DeclareMathOperator{\id}{id}
\DeclareMathOperator{\rank}{rank}
\DeclareMathOperator{\tr}{tr}

\renewcommand{\leq}{\leqslant}
\renewcommand{\geq}{\geqslant}

\begin{document}

\baselineskip=15pt

\title[Affine Yang--Mills--Higgs metrics]{Affine
Yang--Mills--Higgs metrics}

\author[I.\ Biswas]{Indranil Biswas}

\address{School of Mathematics, Tata Institute of Fundamental
Research, Homi Bhabha Road, Bombay 400005, India}

\email{indranil@math.tifr.res.in}

\author[J.\ Loftin]{John Loftin}

\address{Department of Mathematics and Computer Science,
Rutgers University at Newark, Newark, NJ 07102, USA}

\email{loftin@rutgers.edu}

\author[M.\ Stemmler]{Matthias Stemmler}

\address{School of Mathematics, Tata Institute of Fundamental
Research, Homi Bhabha Road, Bombay 400005, India}

\email{stemmler@math.tifr.res.in}



\date{}

\begin{abstract}
Let $(E \, , \varphi)$ be a flat Higgs bundle on a compact
special affine manifold $M$ equipped with an affine Gauduchon
metric. We prove that $(E \, , \varphi)$ is polystable if and
only if it admits an affine Yang--Mills--Higgs metric.
\end{abstract}

\maketitle

\section{Introduction}

An affine manifold of dimension $n$ is a smooth real manifold $M$
of dimension $n$ equipped with a
flat torsion--free connection $D$ on its tangent bundle.
Equivalently, an affine structure on $M$ is provided by an atlas
of $M$ whose transition functions are affine maps of the form $x
\,\longmapsto\, Ax + b$, where $A\, \in\, \text{GL}(n \, , {\mathbb R})$
and $b\, \in\, {\mathbb R}^n$. The total space of the tangent bundle
$TM$ of an affine manifold $M$ admits a natural complex structure;
for the above transition function on $U\, \subset\,{\mathbb R}^n$, the
corresponding
transition map on $TU\, \subset\,T{\mathbb R}^n$ is $z \longmapsto Az +
b$, where $z \,=\, x + \sqrt{-1} \, y$ with $y$ being the fiber
coordinate for the natural trivialization of the tangent bundle
of $U$. There is a
dictionary between the locally constant sheaves on $M$ and the
holomorphic sheaves on $TM$ which are invariant in the fiber
directions (cf.\ \cite{Lo09}). In particular, a flat
complex vector bundle over $M$ naturally extends to a holomorphic
vector bundle over $TM$.

An affine manifold $M$ is called {\em special\/} if it admits a
volume form which is covariant constant with respect to the flat
connection $D$ on $M$. In \cite{Lo09}, a Donaldson--Uhlenbeck--Yau
type correspondence was established for flat vector bundles
over a compact special affine manifold equipped with an
affine Gauduchon metric. This
correspondence states that such a bundle admits an
affine Yang--Mills
metric if and only if it is polystable. The proof of it is an
adaptation to the affine situation of the methods of
Uhlenbeck--Yau \cite{UY86}, \cite{UY89} for the compact K\"ahler
manifolds and their modification by Li--Yau \cite{LY87} for the complex
Gauduchon case.

Hitchin and Donaldson extended the correspondence between
polystable bundles and Yang--Mills connections
to Higgs bundles on Riemann surfaces
\cite{Hi}, \cite{Do}. Simpson extended it to Higgs bundles on
compact K\"ahler manifolds (also to non--compact cases under some
assumption)
using Donaldson's heat flow technique (see \cite{Si88},
\cite{Do85}, \cite{Do87}). Recently, this has been adapted for
the compact Gauduchon case by Jacob \cite{Ja11}.

Our aim here is
to introduce Higgs fields on flat vector bundles over a compact
special affine manifold equipped with an affine Gauduchon metric,
and to establish a correspondence between polystable Higgs
bundles and Yang--Mills--Higgs connections.

We prove the following theorem (see Theorem \ref{main},
Proposition \ref{uniqueness} and Proposition
\ref{stable-simple}):

\begin{theorem}\label{thm0}
Let $M$ be a compact special affine manifold equipped with an
affine Gauduchon metric. If $(E \, , \varphi)$ is a stable
flat Higgs vector bundle over $M$, then $E$ admits an affine
Yang--Mills--Higgs metric, which is unique up to a positive
constant scalar.
\end{theorem}

The analogue of Theorem \ref{thm0} holds for flat real Higgs
bundles (see Corollary \ref{real2}). We also note that Theorem
\ref{thm0} extends to the flat principal Higgs $G$--bundles,
where $G$ is any reductive affine algebraic group over $\mathbb
C$ or of split type over $\mathbb R$; see Section \ref{se-G}.

We recall that a $tt^*$ bundle on a complex manifold $(M\, ,J)$ is a triple
$(E\, ,\nabla\, ,S)$, where $E$ is a $C^\infty$ real vector bundle over $M$,
$\nabla$ is a connection on $E$ and $S$ is a smooth section of $T^*M\otimes
\text{End}(E)$, such that the connection
$$
\nabla^\theta_v \,:=\,\nabla_v +\cos(\theta)\cdot S(v)+ \sin(\theta)\cdot
S(J(v)) \, , \quad v\, \in\, TM
$$
is flat for all $\theta\, \in\, \mathbb R$; see \cite{Sc1}, \cite{Sc2}.
It would be interesting to develop $tt^*$ bundles on affine manifolds.

\medskip
\noindent \textbf{Acknowledgements.}\,  We thank the referee for
helpful comments. The second author is grateful to the Simons
Foundation for support under Collaboration Grant for Mathematicians
210124.

\section{Preparations and statement of the theorem}

Let $M$ be an affine manifold of dimension $n$. As mentioned
before, $TM$ has a natural complex structure. This complex
manifold will be denoted by $M^\CC$. The zero section of
$TM\, =\, M^\CC$ makes $M$ a real submanifold of $M^\CC$.
Given an atlas on $M$ compatible with the affine structure
(so the transition functions are affine maps) the corresponding
coordinates $\{ x^i \}$ are called {\em local affine
coordinates}. If $\{ x^i \}$ is defined on $U\, \subset\, M$,
then
on $TU\, \subset\, TM$, we have the holomorphic coordinate
function $z^i \,:=\, x^i + \sqrt{-1} \, y^i$, where $y^i$
is the fiber coordinate corresponding to the local trivialization
of the tangent bundle given by $\{\frac{\partial}{\partial
x^i}\}_{i=1}^n$.

Define the bundle of $(p \, , q)$ forms on $M$ by
\[
  \A^{p,q} \, := \, \bigwedge\nolimits^p T^\ast M \tensor \bigwedge\nolimits^q T^\ast M \, .
\]
Given local affine coordinates $\{ x^i \}_{i=1}^n$ on $M$, we will
denote the induced frame on $\A^{p,q}$ as
\[
  \big\{ dz^{i_1} \wedge \cdots \wedge dz^{i_p}
  \tensor d\overline{z}^{j_1} \wedge \cdots \wedge d\overline{z}^{j_q}
\big\}
\, ,
\]
where $z^i = x^i + \sqrt{-1} \, y^i$ are the complex coordinates on $M^\CC$ defined above; note that $dz^i = d\bar z^i = dx^i$ on $M$. There is a natural restriction map from $(p \,, q)$--forms on the complex manifold $M^\CC$ to $(p \,, q)$--forms on $M$ given in local affine coordinates on an open subset $U \subset M$ by
\begin{equation} \label{restriction} \begin{split}
  &\sum \, \phi_{i_1, \ldots, i_p, j_1, \ldots, j_q} \, (dz^{i_1} \wedge \cdots \wedge dz^{i_p}) \wedge (d\bar z^{j_1} \wedge \cdots \wedge d\bar z^{j_q}) \\
  \longmapsto\, &\sum \phi_{i_1, \ldots, i_p, j_1, \ldots, j_q}|_U \, (dz^{i_1} \wedge \cdots \wedge dz^{i_p}) \tensor (d\bar z^{j_1} \wedge \cdots \wedge d\bar z^{j_q}) \, ,
\end{split} \end{equation}
where $\phi_{i_1, \ldots, i_p, j_1, \ldots, j_q}$ are smooth functions on $TU \subset TM = M^\CC$, $U$ is considered as the zero section of $TU \longrightarrow U$, and the sums are taken over all $1 \leq i_1 < \cdots < i_p \leq n$ and $1 \leq j_1 < \cdots < j_q \leq n$.

One can define natural operators
\begin{align*}
  \del:  \, \A^{p,q} \,&\longrightarrow\, \A^{p+1,q} \quad \text{ and} \\
  \dbar: \, \A^{p,q} \,&\longrightarrow\, \A^{p,q+1}
\end{align*}
given in local affine coordinates by
\[
  \del\big(\phi \tensor (d\bar z^{j_1} \wedge \cdots \wedge d\bar z^{j_q})\big)
  \,:=\, \frac{1}{2} \, (d \phi) \tensor (d\bar z^{j_1} \wedge \cdots \wedge d\bar z^{j_q})
\]
if $\phi$ is a $p$--form, respectively by
\[
  \dbar\big((dz^{i_1} \wedge \cdots \wedge dz^{i_p}) \tensor \psi\big)
  \,:=\, (-1)^p \frac{1}{2} \, (dz^{i_1} \wedge \cdots \wedge dz^{i_p}) \tensor (d \psi)
\]
if $\psi$ is a $q$--form. These operators are the restrictions of the corresponding operators on $M^\CC$ with respect to the restriction map given in \eqref{restriction}.

Similarly, there is a wedge product defined by
\[
  (\phi_1 \tensor \psi_1) \wedge (\phi_2 \tensor \psi_2) \,:=\, (-1)^{q_1 p_2} \, (\phi_1 \wedge \phi_2) \tensor (\psi_1 \wedge \psi_2)
\]
if $\phi_i \tensor \psi_i$ are forms of type $(p_i \,, q_i)$, $i = 1, 2$; as above, it is the restriction of the wedge product on $M^\CC$.

The tangent bundle $TM$ is equipped with a flat
connection, which we will denote by $D$. The flat connection on
$T^\ast M$ induced by $D$ will be denoted by $D^*$.

The affine manifold $M$ is called {\em special\/} if it admits a volume
(= top--degree) form
$\nu$ which is covariant constant with respect to the flat
connection $D$ on $TM$.

In the case of special affine structures, $\nu$ induces natural
maps
\begin{align*}
  &\A^{n,q} \longrightarrow \bigwedge\nolimits^q T^\ast M, \quad \nu \tensor \chi \longmapsto (-1)^{\frac{n(n-1)}{2}} \chi \, , \\
  &\A^{p,n} \longrightarrow \bigwedge\nolimits^p T^\ast M, \quad \chi \tensor \nu \longmapsto (-1)^{\frac{n(n-1)}{2}} \chi \, ,
\end{align*}
which are called \textit{division\/} by $\nu$. In particular, any
$(n \, , n)$ form $\chi$ can be integrated by considering the integral of $\frac{\chi}{\nu}$. (See \cite{Lo09}.)

A smooth Riemannian metric $g$ on $M$ gives rise to a $(1 \, , 1)$ form
expressed in local affine coordinates as
\begin{equation}\label{deomg}
  \omega_g = \sum_{i,j=1}^n g_{ij} \, dz^i \tensor d\overline{z}^j \, ;
\end{equation}
it is the restriction of the corresponding $(1 \, , 1)$ form on
$M^\CC$ given by the extension of $g$ to $M^\CC$. The metric $g$
is called an {\em affine Gauduchon metric\/} if
$$
\del \dbar (\omega_g^{n-1}) \,=\, 0
$$
(recall that $n\, =\,\dim M$). By \cite[Theorem 5]{Lo09}, on a compact
affine manifold, every conformal class of Riemannian metrics
contains an affine Gauduchon metric, which is unique up to a
positive scalar.

Take a pair $(E\, ,\nabla)$, where $E$ is a complex vector bundle
on $M$, and $\nabla$ is a flat connection on $E$.
(In the following, we will
always be concerned with complex vector bundles until we give
analogues to our results for real vector bundles in
Corollary \ref{real2}.) The pullback of $E$ to $TM\, =\, M^\CC$
by the natural
projection $TM\, \longrightarrow\, M$ will be denoted by
$E^\CC$. The flat connection $\nabla$ pulls back to a
flat connection on $E^\CC$. This flat vector bundle on $M^\CC$
can be considered as an extension of the
flat vector bundle $(E\, ,\nabla)$ on the zero section of $TM$.

Let $h$ be a Hermitian metric on $E$; it defines a
Hermitian metric on the pulled back vector bundle $E^\CC$.
Let $d^h$ be the Chern
connection associated to this Hermitian metric on $E^\CC$. Then
$d^h$ corresponds to a pair $$(\del^h \, , \dbar) \,=\,
(\del^{h,\nabla} \,
, \dbar^{\nabla})$$ of operators on $\A^{p,q}(E) := \A^{p,q}
\tensor E$. This pair is called the {\em extended Hermitian
connection\/} of $(E \, , h)$ (see \cite{Lo09}). Similarly, we
have an {\em extended connection form} $$\theta \,\in\,
\A^{1,0}(\End E)\, ,$$ an {\em extended curvature form\/} $\Omega
\,=\, \dbar \theta \in \A^{1,1}(\End E)$, an {\em extended mean
curvature} $$K\,=\, \tr_g \Omega \,\in\, \A^{0,0}(\End E)$$ and
an {\em extended first Chern form\/} $c_1(E \, , h) \,=\, \tr
\Omega\,\in\, \A^{1,1}$, which are the restrictions of the
corresponding objects on $E^\CC$. Here $\tr_g$ denotes
contraction of differential forms
using the Riemannian metric $g$, and $\tr$ denotes
the trace map on the fibers of $\End E$.

The extended first Chern form $c_1(E \, , h)$ and the extended
mean curvature are related by
\begin{equation} \label{chern-curvature}
  (\tr K) \, \omega_g^n = n \, c_1(E \, , h) \wedge \omega_g^{n-1} \, .
\end{equation}

The {\em degree\/} of a flat vector bundle $E$ over a compact
special affine manifold $M$ equipped with an affine Gauduchon
metric $g$ is defined to be
\begin{equation} \label{degree}
  \deg_g E \, := \, \int_M \frac{c_1(E \, , h) \wedge \omega_g^{n-1}}{\nu} \, ,
\end{equation}
where $h$ is any Hermitian metric on $E$. This is well--defined by
\cite[p.\ 109]{Lo09}.  Even though $E$ admits a flat connection
$\nabla$, there is no reason in general for the degree to be zero in
the Gauduchon case. In particular, we can extend $\nabla$ to a flat
extended connection on $E$ and then define an extended first Chern
form $c_1(E \,, \nabla)$. But
\[
  c_1(E \,, \nabla) - c_1(E \,, h)
  = \tr \dbar \theta_\nabla - \dbar \del \log\det h_{\alpha \bar\beta}
\]
is $\dbar$--exact but not necessarily
$\del\dbar$--exact.  Thus, by integration by parts, the
Gauduchon condition is insufficient to force the degree to be zero.

If ${\rank E} \not=\nolinebreak 0$, the {\em slope\/} of $E$ is
defined as
\[
  \mu_g (E) \, := \, \frac{\deg_g E}{\rank E} \, .
\]

Now we introduce Higgs fields on flat vector bundles.

\begin{definition}\label{dh}
Let $(E \, , \nabla)$ be a smooth vector bundle on $M$ equipped
with a flat connection. A {\em flat Higgs field\/} on $(E \, ,
\nabla)$ is defined to be a smooth section $\varphi$ of
$T^\ast M \tensor \End E$ such that
\begin{enumerate}
\item[(i)] $\varphi$ is covariant constant,
meaning the connection operator
\begin{equation}\label{dna}
\widetilde{\nabla}\, :\, T^\ast M \tensor {\End E}\,
\longrightarrow\, T^\ast M \tensor T^\ast M \tensor {\End E}
\end{equation}
defined by the connections $\nabla$ and $D^*$ on $E$
and $T^\ast M$ respectively, annihilates $\varphi$, and

\item[(ii)] $\varphi \wedge \varphi\, =\, 0$.
\end{enumerate}

If $\varphi$ is a flat Higgs field on $(E \, , \nabla)$, then $(E
\, , \nabla \, , \varphi)$ (or $(E \, , \varphi)$ if $\nabla$ is
understood from the context) is called a {\em flat Higgs bundle}.
\end{definition}

Note that (i) means that the homomorphism
\[
  \varphi\,: \, TM \,\longrightarrow\, \End E
\]
is a homomorphism of flat vector bundles, where $TM$ (respectively,
$\End E$) is equipped with the flat connection $D$ (respectively, the
flat connection induced by the flat connection $\nabla$ on $E$). The
homomorphism $\varphi$ induces a homomorphism
\[
\varphi'\,: \, TM\otimes E \,\longrightarrow\, E\, .
\]
The connections $D$ and $\nabla$ together define a connection on
$TM\otimes E$. The condition (i) means that $\varphi'$ takes
locally defined flat sections of $TM\otimes E$ to locally defined
flat sections of $E$.

Let
\begin{equation}\label{dna2}
d^{\nabla}\, :\, T^\ast M \tensor {\End E}\,
\longrightarrow\, \left(\bigwedge\nolimits^2 T^\ast M\right)
\tensor {\End E}
\end{equation}
be the composition
$$
T^\ast M \tensor {\End E}\,
\stackrel{\widetilde{\nabla}}{\longrightarrow}\,
T^\ast M \tensor T^\ast M \tensor {\End E}
\, \stackrel{{\rm pr}\times\id_{\End E}}{\longrightarrow}
\, \left(\bigwedge\nolimits^2 T^\ast
M\right)\tensor {\End E}\, ,
$$
where ${\rm pr}\, :\, T^\ast M \tensor T^\ast M\, \longrightarrow
\, \bigwedge\nolimits^2 T^\ast M$ is the natural projection, and
$\widetilde{\nabla}$ is defined in \eqref{dna}. So if
$\varphi$ is a flat Higgs field on $(E \, , \nabla)$, then
$d^{\nabla}(\varphi) \,=\, 0$.

The space of all connections on $E$ is an
affine space for the vector space of smooth sections of
$T^\ast M \tensor \End E$; a family of connections $\{ \nabla_t
\}_{t\in\mathbb R}$ is called \textit{affine\/} if
there is a smooth section $\alpha$ of $T^\ast M \tensor \End E$
such that $\nabla_t \,=\, \nabla_0+t\cdot\alpha$.

\begin{lemma}\label{family}
Giving a flat Higgs bundle $(E \, , \nabla \, , \varphi)$ is
equivalent to giving a smooth vector bundle $E$ together with a
$1$--dimensional affine family $\{\nabla_t \,:=\,
\nabla_0+t\cdot\alpha\}_{t\in\mathbb
R}$ of flat connections on $E$ such that the $\End
E$--valued $1$--form $\alpha$ is
flat with respect to the connection on $T^\ast M \tensor {\End
E}$ defined by $\nabla_0$ and $D^*$.
\end{lemma}

\begin{proof}
Given a flat Higgs bundle $(E \, , \nabla \, ,
\varphi)$, we define a family of connections on $E$ by $\nabla_t
\,:= \,\nabla + t \varphi$. In a locally constant frame of $E$
with respect to $\nabla$, we have $d^{\nabla}(\varphi) \,=\, 0$
(see \eqref{dna2} for $d^{\nabla}$) and the
curvature of $\nabla_t$ is as follows:
\begin{equation} \label{curvature}
d^{\nabla}(t \varphi) + (t \varphi) \wedge (t \varphi)\,=\, t \,
d^{\nabla}(\varphi) + t^2 \, \varphi \wedge \varphi \,=\, 0 \, ,
\end{equation}
so $\{ \nabla_t \}_{t \in \RR}$ is a $1$--dimensional affine
family of flat connections on $E$. From the definition of a flat
Higgs field given in Definition \ref{dh} it follows that this
$1$--dimensional affine
family of connections satisfies the condition in the lemma.

For the converse direction, assume that we are given a
$1$--dimensional affine family
of flat connections $\{\nabla_0+t\cdot\alpha\}_{t \in \RR}$ on
$E$, satisfying the condition that $\alpha$ is
flat with respect to the connection on $T^\ast M \tensor {\End
E}$ defined by $\nabla_0$ and $D^*$. Since
$$
0\,=\, d^{\nabla_0}(t \alpha) + (t \alpha) \wedge (t
\alpha)\,=\, t \,
d^{\nabla_0}(\alpha) + t^2 \, \alpha \wedge \alpha\, ,
$$
we conclude that $\alpha \wedge \alpha\,=\, 0$.

Since $\alpha$ is flat with respect to the connection on $T^\ast
M \tensor {\End E}$ defined by $\nabla_0$ and $D^*$, and $\alpha
\wedge \alpha\,=\, 0$, it follows that $(E\, ,\nabla_0\,
,\alpha)$ is a flat Higgs bundle.
\end{proof}

A Higgs field will always be understood as a section of
$\A^{1,0}(\End E)$, meaning it is expressed in local affine
coordinates as
\[
 \varphi \,=\, \sum_{i=1}^n \varphi_i \otimes dz^i\, ,
\]
where $\varphi_i$ are locally defined flat sections of $\End E$;
note that $dz^i\,=\, dx^i$ on $M$.
Given a Hermitian metric $h$
on $E$, the adjoint $\varphi^\ast$ of $\varphi$ with respect to $h$ will be regarded as an element of $\A^{0,1}(\End E)$. In local affine
coordinates, this means that
\[
 \varphi^\ast \,=\, \sum_{j=1}^n (\varphi_j)^\ast \tensor d\overline{z}^j \, .
\]
In particular, the Lie bracket $[\varphi \, , \varphi^\ast]$ is an element of $\A^{1,1}(\End E)$. Locally,
\[
 [\varphi \, , \varphi^\ast]
 \,=\, \varphi \wedge \varphi^\ast + \varphi^\ast \wedge \varphi
 \,=\, \sum_{i,j=1}^n \big(\varphi_i \circ (\varphi_j)^\ast -
(\varphi_j)^\ast \circ
\varphi_i\big) \tensor dz^i \tensor d\overline{z}^j \, .
\]

Let $E$ be a flat vector bundle on $M$ equipped with a flat Higgs
field $\varphi$ as well as a Hermitian metric $h$. The {\em extended
connection form\/} $\theta^\varphi$ of the Hermitian flat Higgs bundle $(E\, ,
\varphi \, , h)$ is defined to be
\[
 \theta^\varphi \, := \, (\theta + \varphi \, , \varphi^\ast) \, \in \, \A^{1,0}(\End E) \oplus \A^{0,1}(\End E) \, ,
\]
where $\varphi^\ast$ denotes the adjoint of $\varphi$ with
respect to $h$. This extended connection form corresponds to
the connection form of $d^h + \varphi + \varphi^\ast$ on $E^\CC
\longrightarrow M^\CC$. Analogously, the {\em extended curvature form\/}
$\Omega^\varphi$ of $(E \, , \varphi \, , h)$ is defined to be
\[
 \Omega^\varphi \, := \, \big(\del^h \varphi \, , \dbar \theta + [\varphi \, , \varphi^\ast] \, , \dbar(\varphi^\ast)\big)
 \, \in \A^{2,0}(\End E) \oplus \A^{1,1}(\End E) \oplus \A^{0,2}(\End E) \, .
\]
It corresponds to the curvature form of the connection $d^h +
\varphi +
\varphi^\ast$ on $E^\CC$. As in the usual case, the {\em extended mean curvature\/} $K^\varphi$ of $(E \, , \varphi \, , h)$ is obtained by contracting the $(1 \, , 1)$ part of the extended curvature $\Omega^\varphi$ using the Riemannian metric $g$, so
\begin{equation}\label{Kvp}
 K^\varphi \, := \, \tr_g \big(\dbar \theta + [\varphi \, , \varphi^\ast]\big) \, \in \, \A^{0,0}(\End E) \, .
\end{equation}
Since $\tr [\varphi \, , \varphi^\ast] = 0$, we have $\tr
K^\varphi = \tr K$, and so by \eqref{chern-curvature}, the
extended mean curvature $K^\varphi$ of $(E \, , \varphi \, , h)$
also is related to the first Chern form $c_1(E \, , h)$ by
\begin{equation} \label{chern-higgs}
 (\tr K^\varphi) \, \omega_g^n \,=\, n \, c_1(E \, , h) \wedge
\omega_g^{n-1} \, .
\end{equation}

\begin{definition}
An {\em affine Yang--Mills--Higgs metric\/} on a flat Higgs bundle
$(E \, , \varphi)$ is a Hermitian metric $h$ on $E$ such that the
extended mean curvature $K^\varphi$ of $(E \, , \varphi \, , h)$ satisfies
the equation
\begin{equation} \label{ymh}
 K^\varphi \,=\, \gamma \cdot\id_E
\end{equation}
for some constant scalar $\gamma$, which is called the {\em Einstein
factor}.
\end{definition}

We show the uniqueness of affine Yang--Mills--Higgs metrics for simple flat Higgs bundles.

\begin{definition}
A flat Higgs bundle $(E \, , \varphi)$ is called {\em simple\/}
if every locally constant section $f$ of $\End E$ satisfying
$[\varphi \, , f] \,=\, 0$ is a constant scalar multiple of the
identity automorphism of $E$.
\end{definition}

\begin{lemma}\label{lemma1}
Let $(E \, , \varphi)$ be a flat Higgs
bundle over a compact affine manifold $M$ equipped with a Riemannian
metric $g$. Assume that $E$ admits an affine Yang--Mills--Higgs metric $h$
with Einstein factor $\gamma$. Let $s$ be a locally constant section of
$E$ with $\varphi(s) = 0$.
\begin{itemize}
\item If $\gamma \,<\, 0$, then $s \,=\, 0$.

\item If $\gamma \,=\, 0$, then $\del^h s \,=\, 0$ and $\varphi^\ast(s)
\,=\, 0$, where $\varphi^\ast$ is the adjoint of $\varphi$ with respect
to $h$.
\end{itemize}
\end{lemma}

\begin{proof}
For any locally constant section $s$ of $E$ with $\varphi(s) = 0$, compute
\[
 \tr_g \del \dbar |s|^2 = - \gamma |s|^2 + |\del^h s|^2 + |\varphi^\ast(s)|^2
\]
and apply the maximum principle.
\end{proof}

\begin{proposition} \label{uniqueness}
Let $(E \, , \varphi)$ be a flat Higgs bundle on a compact affine
manifold $M$ equipped with a Riemannian metric $g$. If $(E \, ,
\varphi)$ is simple, then an affine Yang--Mills--Higgs metric on
$E$ is unique up to a positive scalar.
\end{proposition}

\begin{proof}
Let $h_1$ and $h_2$ be two affine Yang--Mills--Higgs metrics on $E$ with
Einstein factors $\gamma_1$ and $\gamma_2$, respectively. Then there is an endomorphism $f$ of $E$ which is positive definite and self--adjoint with respect to $h_1$ (and $h_2$) such that
\[
 h_2(s \, , t) \,=\, h_1(f(s) \, , t)
\]
for all sections $s$ and $t$ of $E$.

Let $\nabla$ be the flat connection on $E$. Define
\[
 \nabla'\,:=\, f^{\frac{1}{2}} \circ \nabla \circ f^{-
\frac{1}{2}} \quad \text{and} \quad
 \varphi'\,:= \,f^{\frac{1}{2}} \circ \varphi \circ f^{-
\frac{1}{2}} \, .
\]
Then $\nabla'$ is another flat connection on $E$. Denote by $E'$ the new
flat structure on the underlying vector bundle of $E$ induced by
$\nabla'$. Since $\varphi'$ is locally constant with respect to
$\nabla'$, we obtain a new flat Higgs bundle $(E' \, , \varphi')$. The
endomorphism $f^{\frac{1}{2}}$ is a locally constant section of the flat
vector bundle $\Hom(E \, , E')$ and satisfies the equation
$$\varphi_{\Hom}(f^{\frac{1}{2}}) \,=\, 0\, ,$$ where $\varphi_{\Hom}$
is the
flat Higgs field on $\Hom(E \, , E')$ induced by $\varphi$ and
$\varphi'$. We observe that $h_1$ is an affine Yang--Mills--Higgs metric
on $(E' \, , \varphi')$ with Einstein factor $\gamma_2$, and so the
metric $h$ on $\Hom(E \, , E')$ induced by $h_1$ on both $E$ and $E'$ is
an affine Yang--Mills--Higgs metric with Einstein factor $\gamma_2 -
\gamma_1$.

As $f^{\frac{1}{2}} \neq 0$, Lemma \ref{lemma1} implies that
$\gamma_2 - \gamma_1 \geq 0$. By reversing the roles of $h_1$ and $h_2$,
we obtain $\gamma_2 - \gamma_1 = 0$, and so from Lemma \ref{lemma1}
we conclude that $\del^h f^{\frac{1}{2}} \,=\, 0$ and
$\varphi_{\Hom}^\ast(f^{\frac{1}{2}}) \,=\, 0$.

We write $(\del_1 \, , \dbar) = (\del^{h_1,\nabla} \, , \dbar^\nabla)$ and $(\del' \, , \dbar') = (\del^{h_1,\nabla'} \, , \dbar^{\nabla'})$ for the extended Hermitian connections of $(E \, , h_1)$ and $(E' \, , h_1)$, respectively, and calculate
\[
 0 \,=\, \del^h f^{\frac{1}{2}} \,=\, \del' \circ
f^{\frac{1}{2}} -
f^{\frac{1}{2}} \circ \del_1 \,= \,f^{- \frac{1}{2}} \circ \del_1
\circ f - f^{\frac{1}{2}} \circ \del_1 \,=\, f^{- \frac{1}{2}}
\circ \del_1 f \, ,
\]
which implies that $\del_1 f \,=\, 0$. Since $f$ is self--adjoint
with respect to $h_1$, it follows that $\dbar f = 0$.

In an analogous way, we compute
\[
0 \,=\, \varphi_{\Hom}^\ast(f^{\frac{1}{2}})\,=\, (\varphi')^\ast
\circ f^{\frac{1}{2}} - f^{\frac{1}{2}} \circ \varphi^\ast
\,=\, f^{- \frac{1}{2}} \circ \varphi^\ast \circ f -
f^{\frac{1}{2}}\circ \varphi^\ast\,=\, f^{- \frac{1}{2}} \circ
[\varphi^\ast \, , f] \, ,
\]
which implies that $[\varphi^\ast \, , f] \,=\, 0$. Again, since
$f$ is self--adjoint with respect to $h_1$, it follows that
$[\varphi \, , f] \,=\, 0$. As $(E \, , \varphi)$ is simple, $f$
must be a constant scalar multiple of the identity automorphism
of $E$.
\end{proof}

\begin{definition}\label{def stable}
Let $(E \, , \varphi)$ be a flat Higgs bundle on a compact special affine manifold $M$ equipped with an affine Gauduchon metric $g$.
\begin{enumerate}
\item[(i)] $(E \, , \varphi)$ is called {\em stable\/}
(respectively, {\em semistable}) if for every flat subbundle $F$
of $E$ with $0 < \rank F < \rank E$ which is preserved by
$\varphi$, meaning $\varphi(F)\,\subset\,T^\ast M \tensor F$, we
have
\begin{equation} \label{stable}
 \mu_g (F) \,<\, \mu_g (E) \quad
 \text{(respectively, }
 \mu_g (F)\, \leq\, \mu_g (E)
 \text{)} \, .
\end{equation}
\item[(ii)] $(E \, , \varphi)$ is called {\em polystable\/} if
\[
 (E \, , \varphi) \,=\, \bigoplus_{i=1}^N \, (E_i \, , \varphi_i)
\]
with stable flat Higgs bundles $(E_i \, , \varphi_i)$ of the same slope
$\mu_g (E_i)\, =\, \mu_g (E)$.
\end{enumerate}
\end{definition}

\begin{remark}
If $\{ \nabla_t \}_{t \in \RR}$ is the family of flat
connections on $E$ satisfying the condition in Lemma
\ref{family} and corresponding to the flat Higgs bundle
$(E \, , \varphi)$, then
Definition~\ref{def stable} (i) is equivalent to
the condition that \eqref{stable} holds for every smooth subbundle $F$
of $E$ with $0 < \rank F < \rank E$ which is preserved by $\nabla_t$ for all $t$.
\end{remark}

\begin{proposition} \label{stable-simple}
Every stable flat Higgs bundle over a compact special affine manifold is simple.
\end{proposition}

\begin{proof}
Apply the proof of \cite[Proposition 30]{Lo09}, and note that the
condition $[\varphi \, , f]\,= \,0$ implies that the
subbundle $H\,:=\, (f - a
\id_E)(E)$ of $E$ is preserved by $\varphi$.
\end{proof}

We can now state our main theorem.

\begin{theorem} \label{main}
Let $M$ be a compact special affine manifold equipped with an
affine Gauduchon metric $g$. Let $(E \, , \varphi)$ be a stable
flat Higgs vector bundle over $M$. Then $E$ admits an affine
Yang--Mills--Higgs metric.
\end{theorem}

Consider the special case where $\rank E = 1$, meaning $(E \,, \nabla)$ is a flat line bundle over $M$. In this case, the statement of Theorem \ref{main} turns out to be independent of the Higgs field $\varphi$. More precisely, a flat Higgs field on $(E \,, \nabla)$ is nothing but a smooth $1$--form on $M$ which is flat with respect to the flat connection $D^\ast$ on $T^\ast M$. Given a Hermitian metric $h$ on $E$, the extended mean curvature $K^\varphi$ of $(E \,, \varphi \,, h)$ coincides with the usual mean curvature $K$ of $(E \,, h)$, and thus the Yang--Mills--Higgs equation \eqref{ymh} for $(E \,, \varphi \,, h)$ reduces to the usual Yang--Mills equation for $(E \,, h)$. Since, as a line bundle, $E$ is automatically stable, this equation has a solution by \cite[Theorem 1]{Lo09}.

\section{Existence of Yang--Mills--Higgs metrics}

This section is dedicated to the proof of Theorem \ref{main}.

Let $M$ be a compact special affine manifold equipped with a covariant
constant volume form $\nu$ and an affine Gauduchon metric $g$.
Let $(E \, , \varphi)$ be a flat Higgs bundle over $M$. For
any Hermitian metric $h$ on $E$, \eqref{degree} and
\eqref{chern-higgs} together imply that
\[
 \int_M (\tr K^\varphi) \, \frac{\omega_g^n}{\nu}\,=
\,n \, \deg_g E \, ,
\]
where $K^\varphi$ denotes the extended mean curvature of
$(E \, , \varphi \, , h)$. Therefore, the Einstein factor
$\gamma$ of any affine Yang--Mills--Higgs metric on $(E \, , \varphi)$
must satisfy the equation
\begin{equation} \label{einstein-factor}
 \gamma \int_M \frac{\omega_g^n}{\nu} \,=\, n \, \mu_g (E) \, .
\end{equation}

Choose a background Hermitian metric $h_0$ on $E$. Any Hermitian metric $h$ on $E$ is represented by an endomorphism $f$ of $E$ such that
\[
 h(s \, , t) \,=\, h_0(f(s) \, , t)
\]
for all sections $s$ and $t$ of $E$. This endomorphism $f$ is positive
definite and self--adjoint with respect to $h_0$. As we pass from
$h_0$ to $h$, the extended connection form, curvature form and mean
curvature change as follows:
\begin{align}
 \label{connection form}   \theta^\varphi &=\, \theta_0^\varphi
+ \big(f^{-1} \del_0 f \, , f^{-1} [\varphi^\ast, f]\big) \, , \\
  \label{curvature form}     \Omega^\varphi &=\, \Omega_0^\varphi
+ \big([f^{-1} \del_0 f \, , \varphi] \, , \dbar(f^{-1} \partial_0 f) + [\varphi \, , f^{-1} [\varphi^\ast \, , f]] \, , \dbar (f^{-1} [\varphi^\ast \, , f])\big) \, , \\
  \label{mean curvature}     K^\varphi      &=\, K_0^\varphi +
\tr_g \dbar(f^{-1} \del_0 f) + \tr_g [\varphi \, , f^{-1} [\varphi^\ast \, , f]] \, , \\
  \label{tr(mean curvature)} \tr K^\varphi  &=\, \tr K_0^\varphi
- \tr_g \del \dbar \log (\det f) \, .
\end{align}
Here, $\theta^\varphi$, $\Omega^\varphi$ and $K^\varphi$ are
defined with respect to $h$, and $\theta_0^\varphi$,
$\Omega_0^\varphi$ and $K_0^\varphi$ are defined with respect to
$h_0$. Moreover, $(\del_0, \dbar) \,= \,(\del^{h_0}, \dbar)$
denotes the extended Hermitian connection on $(E \, , h_0)$, and
$\varphi^\ast$ is the adjoint of $\varphi$ with respect to $h_0$.

According to \eqref{mean curvature}, we need to solve the equation
\[
 K_0^\varphi - \gamma \id_E + \tr_g \dbar (f^{-1} \del_0 f) +
\tr_g [\varphi \, , f^{-1} [\varphi^\ast \, , f]] \,=\, 0 \, ,
\]
where $\gamma$ is determined by \eqref{einstein-factor}.

As done in the usual case, we will solve this equation by the continuity
method. For $\eps \in [0 \, , 1]$, consider the equation
\begin{equation}\label{equation1}
  L_\eps(f) \, := \, K_0^\varphi - \gamma \id_E + \tr_g \dbar (f^{-1}
\del_0 f) + \tr_g [\varphi \, , f^{-1} [\varphi^\ast \, , f]] + \eps
\log f = 0\, ,
\end{equation}
and let
\[
  J \, := \, \big\{ \eps \in (0 \, , 1]\;\big|\; \text{there is a
smooth solution } f \text{ to } L_\eps(f) = 0 \big\} \, .
\]
We will use the continuity method to show that $J \,=\, (0 \, ,
1]$
for any simple flat Higgs bundle $(E \, , \varphi)$, and then show
that we may take $\eps\,\longrightarrow\, 0$ to get a limit
of solutions if $(E \, ,
\varphi)$ is stable.
Note that by Proposition \ref{stable-simple}, if $(E \, , \varphi)$ is stable, then it is automatically simple.

The first step in the continuity method is to show that $1 \in J$ and so
$J$ is non--empty. The following proposition also yields, apart from
the above mentioned inclusion, an appropriately
normalized background metric $h_0$ on $E$.

\begin{proposition} \label{background-metric}
There is a smooth Hermitian metric $h_0$ on $E$ such that the
equation $L_1(f)\,=\, 0$ has a smooth solution $f_1$. The metric
$h_0$ satisfies the normalization $\tr K_0^\varphi = \nolinebreak r \gamma$,
where $r$ is the rank of $E$, and $\gamma$ is given by
\eqref{einstein-factor}.
\end{proposition}

\begin{proof}
As we have $\tr K^\varphi\,=\,\tr K$ for the extended mean
curvature of any Hermitian metric on $E$, the proof of \cite[Proposition 7]{Lo09} also works for Higgs bundles.
\end{proof}

So we choose $h_0$ according to Proposition
\ref{background-metric} and obtain the following:

\begin{corollary} \label{initial-solution}
The inclusion $1 \,\in\, J$ holds.
\end{corollary}

\subsection{Openness of $J$}

Let $\Herm(E \, , h_0)$ be the real vector bundle of endomorphisms of
$E$ which are self--adjoint with respect to $h_0$. For any
Hermitian metric $h$ on $E$, we know that
$[\varphi \, , \varphi^\ast]$ is anti--self--adjoint.
Therefore, as in \cite[Lemma 3.2.3]{LT95}, for any $f
\,\in\, \Herm(E \, , h_0)$, we have
\begin{equation}\label{Lef}
  \widehat L(\eps \, , f) \, := \, f L_\eps(f) = f K^\varphi - \gamma f + \eps f \log f \, \in \, \Herm(E \, , h_0) \, .
\end{equation}

Let $1 < p < \infty$, and let $k$ be a sufficiently large
integer.

Assume that $\eps \,\in\, J$, meaning there is a smooth solution
$f_\eps$ to $L_\eps(f) \,=\, 0$, or equivalently $\widehat L(\eps
\, , f) \,=\, 0$. We will use the implicit function theorem to
show that there is some $\delta\,> \,0$ such that for every
$\eps'\,\in\, (\eps - \delta \, , \eps + \delta)$, there is a
solution $f_{\eps'}$ to $\widehat L(\eps' \, , f) \,=\, 0$
lying in $L_k^p \Herm(E \, , h_0)$. By choosing $k$ large
enough, it then follows that each $f_{\eps'}$ is smooth. Thus
$(\eps - \delta \, , \eps + \delta) \cap (0 \, , 1] \,\subset\,
J$, implying that $J$ is open.

In order to be able to apply the implicit function theorem, we
have to show that
\begin{equation}\label{dXi}
  \Xi \, := \, \frac{\delta}{\delta f} \, \widehat L(\eps \, ,
f)\,:\, L_k^p \Herm(E \, , h_0)\,\longrightarrow\, L_{k-2}^p
\Herm(E \, , h_0)
\end{equation}
is an isomorphism of Banach spaces. For $\phi\,\in\, \Herm(E \, ,
h_0)$, the Higgs field $\varphi$ does not contribute any
derivatives of $\phi$ to $\Xi(\phi)$. So the following lemma from
\cite{Lo09} is still valid for Higgs bundles (see \cite[Lemma
9]{Lo09}):

\begin{lemma}
The linear operator $\Xi$ in \eqref{dXi} is elliptic Fredholm of index
$0$.
\end{lemma}

Consequently, in order to be able to apply the implicit function
theorem, it is enough to show that $\Xi$ is injective.

As in \cite[p.\ 116]{Lo09}, for an endomorphism $f$ of $E$ which is positive definite and self--adjoint with respect to $h_0$, define
\[
  \del_0^f \, := \, \Ad f^{-\frac{1}{2}} \circ \del_0 \circ \Ad f^{\frac{1}{2}} \quad \text{and} \quad
  \dbar^f  \, := \, \Ad f^{\frac{1}{2}}  \circ \dbar  \circ \Ad f^{-\frac{1}{2}}
\]
and also
\begin{equation}\label{vpf}
  \varphi^f \, := \, (\Ad f^{\frac{1}{2}})(\varphi) \, ,
\end{equation}
where
\[
  (\Ad s)(\psi) \, := \, s \circ \psi \circ s^{-1}
\]
for an automorphism $s$ and an endomorphism $\psi$ of $E$.

\begin{proposition} \label{estimate1}
Let $\alpha \in \RR$ and $\eps \in (0 \, , 1]$. Let $f$ be an
endomorphism of $E$ which is positive definite and self--adjoint
with respect to $h_0$, and let $\phi \in \Herm(E \, , h_0)$. Assume that
$\widehat L(\eps \, , f) \,=\, 0$ (see \eqref{Lef}) and
\begin{equation} \label{assumption}
  \frac{\delta}{\delta f} \, \widehat L(\eps \, , f)(\phi) + \alpha f
\log f \,=\, \Xi(\phi) + \alpha f \log f \,=\, 0 \, ,
\end{equation}
where $\Xi$ is defined in \eqref{dXi}. Then for $\eta \,:=\,
f^{-\frac{1}{2}} \circ \phi \circ f^{-\frac{1}{2}}$, we have
\[
  - \tr_g \del \dbar |\eta|^2 + 2 \eps |\eta|^2 + |\del_0^f
\eta|^2 + |\dbar^f \eta|^2 + 2 \big|[\varphi^f \, , \eta]\big|^2
\,\leq\, - 2 \alpha h_0(\log f \, , \eta) \, .
\]
\end{proposition}

\begin{proof}
By definition of $\widehat L$, we have
\[
  \Xi(\phi)\,=\, \phi \circ L_\eps(f) + f \circ
\frac{\delta}{\delta f} \, L_\eps(f)(\phi) \, .
\]
The first term vanishes because $\widehat L(\eps \, , f)\,=\, 0$.
{}From \eqref{assumption} it follows that
\[
\frac{\delta}{\delta f}\, L_\eps(f)(\phi)\,=\, -\alpha \log f\, .
\]
The left--hand side can be computed as in \cite[proof of Proposition 3.2.5]{LT95}. The additional contribution due to the Higgs field is as follows:
\begin{align*}
  \left. \frac{d}{dt} \, \tr_g [\varphi \, , (f + t \phi)^{-1} [\varphi^\ast \, , f + t \phi]] \right|_{t=0}
  &=\, \tr_g [\varphi \, , [f^{-1} \varphi^\ast f \, , f^{-1}
\phi]] \\
  &=\, f^{-\frac{1}{2}} \circ \tr_g [\varphi^f \, ,
[(\varphi^f)^\ast \, , \eta]] \circ f^{\frac{1}{2}} \, .
\end{align*}
Following \cite{LT95}, we write
\[
  P^f \, := \, \tr_g \dbar^f \del_0^f \quad \text{and} \quad
 \Phi \, := \, f^{\frac{1}{2}} \circ \frac{\delta}{\delta f} \,
(\log f)(\phi) \circ f^{-\frac{1}{2}}\, ,
\]
and obtain
\[
 P^f(\eta) + \tr_g [\varphi^f \, , [(\varphi^f)^\ast \, , \eta]]
+ \eps \Phi \,=\, - \alpha \log f \, .
\]
We compute
\[
  \tr_g \big(h_0\big([\varphi^f \, , [(\varphi^f)^\ast \, , \eta]] \, , \eta\big)
  + h_0\big(\eta \, , [\varphi^f \, , [(\varphi^f)^\ast \, , \eta]]^\ast\big)\big)
  \,=\, 2 \big|[\varphi^f\, , \eta]\big|^2\, ,
\]
and together with the estimates in \cite{LT95}, the proposition
follows.
\end{proof}

\begin{proposition} \label{openness}
The subset $J$ is open.
\end{proposition}

\begin{proof}
We show that $\Xi$ is injective. Take any $\phi$ such
that $\Xi(\phi) \,=\, 0$. Setting $\alpha\, =\, 0$ in
Proposition \ref{estimate1} we see that
\[
  - \tr_g \del \dbar |\eta|^2 + 2 \eps |\eta|^2 \leq 0 \, .
\]
Therefore, the maximum principle gives that $|\eta|^2 \,=\, 0$.
So $\phi \,=\, 0$, proving that $\Xi$ is injective.
As explained before, this completes the proof of Proposition
\ref{openness}.
\end{proof}

\subsection{Closedness of $J$}

As in \cite[Lemma 12]{Lo09}, we have the following:
\begin{lemma} \label{det}
Let $f$ be an endomorphism of $E$ which is positive definite
and self--adjoint with respect to $h_0$. If $L_\eps(f) \,= \,0$
(defined in \eqref{equation1}) for some
$\eps \,>\, 0$, then $\det f \,=\, 1$.
\end{lemma}

Let
\begin{equation}\label{14}
f \,=\, f_\eps
\end{equation}
be the family of solutions constructed for $\eps \in (\eps_0 \, , 1]$
in Corollary \ref{initial-solution} and Proposition \ref{openness}. Define
\begin{equation}\label{15}
  m \, := \, m_\eps \, := \, \max_M |\log f_\eps| \, , \quad
  \phi \, := \, \phi_\eps \, := \, \frac{d f_\eps}{d \eps} \, , \quad
  \eta \, := \, \eta_\eps \, := \, f_\eps^{-\frac{1}{2}} \circ \phi_\eps \circ f_\eps^{-\frac{1}{2}} \, .
\end{equation}

As in \cite[Lemma 13]{Lo09}, Lemma \ref{det} implies the
following:

\begin{lemma} \label{trace}
For $\eta_\eps$ in \eqref{15},
$$\tr \eta_\eps \,=\, 0\, .$$
\end{lemma}

On $M$, consider the $L^2$ inner products on $\A^{p,q}(\End E)$
given by $h_0$, $g$ and the volume form $\frac{\omega_g^n}{\nu}$.
Then we have the following proposition.

\begin{proposition} \label{estimate2}
Let $(E \, , \varphi)$ be a simple flat Higgs bundle over $M$.
Let $f$ be as in \eqref{14}.
Then there is a constant $C(m)$ depending only on $M$, $g$,
$\nu$, $\varphi$, $h_0$, and $m\,=\, m_\eps$ such that for $\eta
\,=\,\eta_\eps$, we have
\[
 ||\dbar^f \eta||_{L^2}^2 + \big\lVert [\varphi^f \, , \eta]
\big\rVert_{L^2}^2 \,\geq\, C(m) ||\eta||_{L^2}^2 \, ,
\]
where $\varphi^f$ is defined in \eqref{vpf}.
\end{proposition}

\begin{remark}\label{rC}
Following \cite{Lo09}, henceforth $C(m)$ will always denote a
constant depending on $M$, $g$, $\nu$, $\varphi$, $h_0$ and $m$.
However, the particular constant may change with the context.
Similarly, $C$ will denote a constant depending only on the
initial data $M$, $g$, $\nu$, $\varphi$ and $h_0$, but not on
$\eps$ or $m$.
\end{remark}

\begin{proof}
Let $\psi\,:=\, f^{-\frac{1}{2}} \circ \eta \circ
f^{\frac{1}{2}}$.
Then pointwise, we have
\[
  |\dbar^f \eta|^2 + \big| [\varphi^f \, , \eta] \big|^2
  = |f^{\frac{1}{2}} \circ \dbar \psi \circ f^{-\frac{1}{2}}|^2 + \big|f^{\frac{1}{2}} \circ [\varphi \, , \psi] \circ f^{-\frac{1}{2}}\big|^2
\geq C(m) \big( |\dbar \psi|^2 + \big| [\varphi \, , \psi]
\big|^2 \big) \, .
\]
Integrating both sides of it over $M$ with respect to the volume form
$\frac{\omega_g^n}{\nu}$, we obtain
\begin{equation} \label{inequality}
  ||\dbar^f \eta||_{L^2}^2 + \big\lVert [\varphi^f \, , \eta]
\big\rVert_{L^2}^2\,\geq\, C(m) \big( ||\dbar \psi||_{L^2}^2 +
\big\lVert [\varphi \, , \psi] \big\rVert_{L^2}^2 \big) \, .
\end{equation}

The space $\A^{1,0}(\End E) \oplus \A^{0,1}(\End E)$ on $M$
corresponds to the space of $1$--forms on $M^\CC$ with values in
$\End E^\CC$. It has a natural $L^2$ inner product induced by the
$L^2$ inner products on $\A^{1,0}(\End E)$ and $\A^{0,1}(\End E)$. Consider the operator
\[
  L: \A^{0,0}(\End E) \longrightarrow \A^{1,0}(\End E) \oplus \A^{0,1}(\End E), \quad
  \chi \longmapsto \big([\varphi \, , \chi] \, , \dbar \chi\big) \, .
\]
Its adjoint with respect to the $L^2$ inner products is
\[
  L^\ast: \A^{1,0}(\End E) \oplus \A^{0,1}(\End E) \longrightarrow \A^{0,0}(\End E), \quad
  (u \, , v) \longmapsto \tr_g [\varphi^\ast \, , u] + \dbar^\ast v \, .
\]
Using this, the right--hand side in \eqref{inequality} can be written as
\[
  C(m) ||L \psi||_{L^2}^2 \,=\, C(m) \left<L^\ast L \psi \, ,
\psi\right>_{L^2} \, .
\]
The operator $L^\ast L$ is self--adjoint, and it is elliptic because
$L^\ast L \chi$ is equivalent to $\dbar^\ast \dbar \chi$ up to
zeroth--order derivatives of $\chi$. For any $\chi$ in the kernel
of $L^\ast L$, we have
\[
  0 \,=\, \left<L^\ast L \chi \, , \chi\right>_{L^2} \,=\, ||L
\chi||_{L^2}^2 \,=\, ||\dbar \chi||_{L^2}^2 + \big\lVert [\varphi \,
, \chi] \big\rVert_{L^2}^2 \, ,
\]
and so $\chi$ is a locally constant section of $\End E$
satisfying $[\varphi \, , \chi] \,=\, 0$. Since $(E \, ,
\varphi)$ is simple, it follows that the kernel of $L^\ast L$
consists only of the constant multiples of the identity
automorphism. As in
\cite[proof of Proposition 14]{Lo09}, Lemma \ref{trace} implies
that $\psi$ is $L^2$--orthogonal to the kernel of $L^\ast L$, and
hence there is a constant $\lambda_1 \,>\, 0$ (the smallest
positive eigenvalue of $L^\ast L$) such that
\[
  \left<L^\ast L \psi \, , \psi\right>_{L^2}\,\geq\, \lambda_1
||\psi||_{L^2}^2 \, .
\]
Combining this with the inequality in \eqref{inequality}, it now
follows that
\[
  ||\dbar^f \eta||_{L^2}^2 + \big\lVert [\varphi^f \, , \eta]
\big\rVert_{L^2}^2\,
\geq\, C(m) \left<L^\ast L \psi \, , \psi\right>_{L^2}
\,\geq\, C(m) ||\psi||_{L^2}^2 \,\geq\, C(m) ||\eta||_{L^2}^2 \,
. \qedhere
\]
\end{proof}

\begin{proposition} \label{phi-bound}
Let $(E \, , \varphi)$ be a simple flat Higgs bundle over $M$. Then
\[
  \max_M |\phi_\eps| \leq C(m) \, ,
\]
where $\phi_\eps$ is defined in \eqref{15}.
\end{proposition}

\begin{proof}
This follows as in \cite[Proposition 16]{Lo09} from
Proposition \ref{estimate1} and Proposition \ref{estimate2}.
\end{proof}

\begin{lemma} \label{lemma2}
Let $f$ be as in Proposition \ref{estimate2}. Then
\[
  - \frac{1}{2} \, \tr_g \del \dbar |\log f|^2 + \eps |\log f|^2
\,\leq\, |K_0^\varphi - \gamma \id_E| \cdot |\log f| \, ,
\]
where $K_0^\varphi$ is defined as in \eqref{Kvp} for $h_0$
(see \eqref{tr(mean curvature)}).
\end{lemma}

\begin{proof}
Since $L_\eps(f) \,=\, 0$ (see equation \eqref{equation1}),
we have
\begin{equation} \label{equation2}
  K_0^\varphi - \gamma \id_E \,=\, - \tr_g \dbar (f^{-1} \del_0
f)
-
\tr_g [\varphi \, , f^{-1} [\varphi^\ast \, , f]] - \eps \log f \, .
\end{equation}
This implies that
\begin{align*}
  |K_0^\varphi - \gamma \id_E| \cdot |\log f|
  &\geq \big|h_0(- (K_0^\varphi - \gamma \id_E) \, , \log f)\big| \\
  &\geq h_0(\tr_g \dbar (f^{-1} \del_0 f) + \eps \log f \, , \log f)
  + h_0(\tr_g [\varphi \, , f^{-1} [\varphi^\ast \, , f]] \, , \log f)
\end{align*}
if both summands on the right--hand side are real. From
\cite[proof of Lemma 3.3.4~(i)]{LT95}, we know that the first
summand is real and satisfies the condition
\[
  h_0(\tr_g \dbar (f^{-1} \del_0 f) + \eps \log f \, , \log f)
\,\geq\, - \frac{1}{2} \, \tr_g \del \dbar |\log f|^2 + \eps
|\log f|^2 \, .
\]

So to complete the proof of the proposition, it suffices to show
that
\begin{equation}\label{shiq}
h_0(\tr_g
[\varphi \, , f^{-1} [\varphi^\ast \, , f]] \, , \log f)\,\in\,
{\mathbb R}^{\geq 0}\, .
\end{equation}

The argument for it is similar to the one in \cite{LT95}. Over
each point of $M$, we can write
\[
f \,=\, \sum_{\alpha=1}^r \exp(\lambda_\alpha) \, e_\alpha
\tensor e^\alpha
\]
in a $h_0$--unitary frame $\{ e_\alpha \}$ of $E$, where $r$ is
the rank of $E$, and $\{ e^\alpha \}$ is the dual frame of
$E^\ast$; the eigenvalues $\lambda_\alpha$ are real. Then we have
\[
  \log f \,=\, \sum_{\alpha=1}^r \lambda_\alpha \, e_\alpha
\tensor e^\alpha \, ,
\]
and writing
\[
\varphi\,=\, \sum_{\alpha,\beta=1}^r \varphi^\alpha_\beta \,
e_\alpha \tensor e^\beta \, ,
\]
we compute
\begin{align*}
  h_0(\tr_g [\varphi \, , f^{-1} [\varphi^\ast \, , f]] \, , \log f)
  &= - \tr_g \tr (f^{-1} [\varphi^\ast \, , f] \wedge [\varphi^\ast \, , \log f]^\ast) \\
  &= - \tr_g \sum_{\alpha,\beta=1}^r (\exp(\lambda_\alpha - \lambda_\beta) - 1) \, \overline{\varphi^\alpha_\beta} \wedge (\lambda_\alpha - \lambda_\beta) \, \varphi^\alpha_\beta \\
  &= \sum_{\alpha,\beta=1}^r (\exp(\lambda_\alpha - \lambda_\beta) - 1)(\lambda_\alpha - \lambda_\beta) |\varphi^\alpha_\beta|^2 \, .
\end{align*}
Therefore, \eqref{shiq} holds because
$x(\exp(x) - 1) \,\in\, {\mathbb R}^{\geq 0}$ for all $x \in \RR$. We
already noted that
\eqref{shiq} completes the proof of the proposition.
\end{proof}

The following corollary can be derived from Lemma \ref{lemma2} as in \cite[Corollaries 18 and~19]{Lo09}.

\begin{corollary} \label{corollary1} \mbox{}
\begin{enumerate}
\item[(i)]  $m \,\leq\, \eps^{-1} C$, where $m$ is defined in
\eqref{15}, and $C$ is as in Remark \ref{rC}, and
\item[(ii)] $m \,\leq\, C \big(||\log f||_{L^2} + 1\big)^2$.
\end{enumerate}
\end{corollary}

\begin{proposition}
Let $(E \, , \varphi)$ be a simple flat Higgs bundle over $M$.
Suppose there is an $m \,\in\, \RR$ such that $m_\eps\,\leq\, m$
for all $\eps\,\in\,(\eps_0 \, , 1]$. Let $\phi_\eps$ and $f_\eps$
be as in \eqref{15}. Then for all $p > 1$ and
$\eps \in (\eps_0 \, , 1]$,
\[
||\phi_\eps||_{L_2^p}\,\leq\, C(m) (1 + ||f_\eps||_{L_2^p}) \, ,
\]
where $C(m)$ may depend on $p$ as well as $m$ along with the
initial data.
\end{proposition}

\begin{proof}
We proceed as in the proof of \cite[Proposition 21]{Lo09}. Similar to \cite[equation (19)]{Lo09}, for the operator
\[
  \Lambda \, := \,  n \, \del_0^\ast \del_0 + \id_E \, ,
\]
we obtain
\begin{equation} \label{equation3} \begin{split}
\Lambda \phi \,=\, & - \phi \big(K_0^\varphi - (\gamma + 1) \id_E + \eps
\log f + \tr_g [\varphi \, , f^{-1} [\varphi^\ast \, , f]]\big) \\
               & - \tr_g(\dbar f \wedge f^{-1} \phi f^{-1} \del_0 f) + \tr_g(\dbar f \wedge f^{-1} \del_0 \phi)
                 + \tr_g(\dbar \phi \wedge f^{-1} \del_0 f) \\
               & - f \log f - \eps f \left(\frac{\delta}{\delta f} \, \log f\right)(\phi)
                 - n \, \frac{\del_0 \phi \wedge \dbar \omega_g^{n-1}}{\omega_g^n} \\
               & - f \tr_g [\varphi \, , [f^{-1} \varphi^\ast f \, , f^{-1} \phi]] \, .
\end{split} \end{equation}
Compared to \cite[equation (19)]{Lo09}, the right--hand side of equation \eqref{equation3} contains the two additional terms
\[
  - \phi \tr_g [\varphi \, , f^{-1} [\varphi^\ast \, , f]] \quad \text{and} \quad
  - f \tr_g [\varphi \, , [f^{-1} \varphi^\ast f \, , f^{-1} \phi]] \, .
\]
By Proposition \ref{phi-bound}, these are both bounded in $L^p$
norm by $C(m)$. Consequently, the proposition follows as in
\cite{Lo09}.
\end{proof}

As in \cite[Corollary 22]{Lo09}, we obtain the following corollary.
\begin{corollary} \label{corollary2}
Suppose there is an $m\,\in\, \RR$ such that $m_\eps\,\leq\,m$
for all $\eps \,\in\, (\eps_0 \, , 1]$. Then for all $\eps\,\in\,
(\eps_0 \, , 1]$, we have $||f_\eps||_{L_2^p} \,\leq\, C(m)$,
where $C(m)$ is independent of $\eps$.
\end{corollary}

\begin{proposition}
Let $(E \, , \varphi)$ be a simple flat Higgs bundle over $M$. Then
\begin{enumerate}
\item[(i)]  $J = (0 \, , 1]$, and
\item[(ii)] if $||f_\eps||_{L^2}$ (see \eqref{14}) is bounded
independently of $\eps \in
(0 \, , 1]$, then there exists a smooth solution $f_0$ to the Yang--Mills--Higgs equation $L_0(f) = 0$.
\end{enumerate}
\end{proposition}

\begin{proof}
For (i), it is enough to show that if $J \,=\, (\eps_0 \, , 1]$ for
$\eps_0 \,\in\, (0 \, , 1)$, then there is a smooth solution
$f_{\eps_0}$ to $L_{\eps_0}(f) \,=\, 0$. Indeed, this implies
that $J$ is closed and so (i) follows from Corollary
\ref{initial-solution} and Proposition \ref{openness}.

For (ii),
we need to show the same for $\eps_0 \,=\, 0$.
In both cases, we know that there is a constant $C \,>\, 0$ such
that $||f_\eps||_{L_2^p} \,\leq\, C$ for all $\eps \,\in\,
(\eps_0 \, , 1]$. Indeed, in case of (i), this follows from Corollary
\ref{corollary1} (i) and Corollary \ref{corollary2}, and in case
of (ii), it follows from Corollary \ref{corollary1} (ii),
Corollary \ref{corollary2} and the hypothesis of (ii).

So assume that $J \,=\, (\eps_0 \, , 1]$ for $\eps_0 \,\in\, [0
\, , 1)$
and that there is a constant $C \,>\, 0$ such that
$||f_\eps||_{L_2^p} \,\leq\, C$ for all $\eps\,\in\, (\eps_0 \, ,
1]$. We will find a sequence $\eps_i\,\longrightarrow\, \eps_0$
such that the limit $f_{\eps_0} \,=\, \lim_{i\to\infty} f_{\eps_i}$ is
the required solution.

Choose $p \,>\, n$. Then $L_1^p$ maps compactly into $C^0$. The
uniform $L_2^p$ norm bound implies that there is a sequence $\eps_i
\,\longrightarrow\,\eps_0$ such that
$f_{\eps_i}\,\longrightarrow\, f_{\eps_0}$
converges weakly in $L_2^p$ norm and strongly in $L_1^p$ norm as
well as in $C^0$ norm.

For a smooth section $\alpha$ of $\End E$, we compute in the sense of distributions:
\begin{align*}
  \left<L_{\eps_0}(f_{\eps_0}) \, , \alpha\right>_{L^2}
  &\,=\, \left<L_{\eps_0}(f_{\eps_0}) - L_{\eps_i}(f_{\eps_i}),
\alpha\right>_{L^2} \\
  &=\, \int_M h_0\big(\tr_g \dbar(f_{\eps_0}^{-1} \del_0 f_{\eps_0} -
f_{\eps_i}^{-1} \del_0 f_{\eps_i}) \, , \alpha\big) \, \frac{\omega_g^n}{\nu} \\
  &\phantom{=} + \int_M h_0(\eps_0 \log f_{\eps_0} - \eps_i \log f_{\eps_i} \, , \alpha) \, \frac{\omega_g^n}{\nu} \\
  &\phantom{=} + \int_M h_0\big(\tr_g [\varphi \, , f_{\eps_0}^{-1} [\varphi^\ast \, , f_{\eps_0}] - f_{\eps_i}^{-1} [\varphi^\ast \, , f_{\eps_i}]] \, , \alpha\big) \, \frac{\omega_g^n}{\nu} \, .
\end{align*}
The first two integrals go to zero as $i\,\longrightarrow\,
\infty$ by \cite[proof of Proposition 23]{Lo09}. For the third
integral, we can assume that $f_{\eps_i}^{-1}\,\longrightarrow\,
f_{\eps_0}^{-1}$
strongly in $L_1^p$ norm and thus in $C^0$ norm (after going to a
subsequence) since $f^{-1} \,= \, \exp(- \log f)$, and both
$\exp$ and $\log$ maps on functions are continuous in $L_1^p$
norm. As $f_{\eps_i}\, \longrightarrow\,
f_{\eps_0}$ in $C^0$ norm, the third integral also
goes to zero as $i\,\longrightarrow\, \infty$. Therefore,
$L_{\eps_0}(f_{\eps_0})\,=\, 0$ in the sense of distributions.

In the same way, it can be shown that for $f_{\eps_0} \,\in\, L_2^p$,
we have $\tr_g \dbar \del_0 f_{\eps_0} \,\in\, L_1^p$. As in
\cite{Lo09}, it then follows that $f_{\eps_0}$ is smooth
and satisfies the equation $L_{\eps_0}(f) \,=\, 0$.
\end{proof}

\subsection{Construction of a destabilizing subbundle}

We will construct a destabilizing flat subbundle of $(E \, ,
\varphi)$ if $\limsup_\eps ||f_\eps||_{L^2} \,=\, \infty$. For a
sequence $\eps_i\,\longrightarrow\, 0$, we will re--scale by the
reciprocal $\rho_i$
of the largest eigenvalue of $f_{\eps_i}$. Then we will show that
the limit
\[
  \lim_{\sigma \to 0} \, \lim_{i \to \infty} \, (\rho_i f_{\eps_i})^\sigma
\]
exists, and each of its eigenvalues is $0$ or $1$. A projection to
the destabilizing subbundle will be given by $\id_E$ minus this limit.

\begin{proposition} \label{estimate-destabilizing}
Let $\eps \,>\, 0$ and $0 \,<\, \sigma \,\leq\, 1$. If
$L_\eps(f)\,=\,0$, then
\[
  - \frac{1}{\sigma} \, \tr_g \del \dbar(\tr f^\sigma) + \eps h_0(\log f \, , f^\sigma) + \big|f^{-\frac{\sigma}{2}} \del_0(f^\sigma)\big|^2
 + \big|f^{-\frac{\sigma}{2}} [\varphi^\ast \, , f^\sigma]\big|^2 \leq -
h_0(K_0^\varphi - \gamma \id_E \, , f^\sigma) \, ,
\]
where $f$ is as in \eqref{14} and $K_0^\varphi$ is defined as in
\eqref{Kvp} for $h_0$.
\end{proposition}

\begin{proof}
Using \eqref{equation2}, we have
\[
  - h_0\big(K_0^\varphi - \gamma \id_E \, , f^\sigma\big)
\,=\, h_0\big(\tr_g \dbar (f^{-1} \del_0 f) + \eps \log f \, ,
f^\sigma\big)
    + h_0\big(\tr_g [\varphi \, , f^{-1} [\varphi^\ast \, , f]] \, , f^\sigma\big) \, .
\]
By \cite[proof of Lemma 3.4.4 (ii)]{LT95}, the first summand satisfies
\[
  h_0\big(\tr_g \dbar (f^{-1} \del_0 f) + \eps \log f \, , f^\sigma\big)
\,\geq\, - \frac{1}{\sigma} \, \tr_g \del \dbar(\tr f^\sigma) +
\eps
h_0(\log f \, , f^\sigma) + \big|f^{-\frac{\sigma}{2}} \del_0(f^\sigma)\big|^2 \, .
\]
It remains to show that
\begin{equation}\label{rs}
h_0\big(\tr_g [\varphi, f^{-1} [\varphi^\ast \, , f]] \, ,
f^\sigma\big)\,\geq\, \big|f^{-\frac{\sigma}{2}} [\varphi^\ast \,
, f^\sigma]\big|^2\, .
\end{equation}
In the notation of Lemma \ref{lemma2}, we have
\begin{align*}
  h_0\big(\tr_g [\varphi \, , f^{-1} [\varphi^\ast \, , f]] \, , f^\sigma\big)
  &= - \tr_g \tr (f^{-1} [\varphi^\ast\, , f] \wedge [\varphi^\ast \, , f^\sigma]^\ast) \\
  &= - \tr_g \sum_{\alpha,\beta=1}^r (\exp(\lambda_\alpha - \lambda_\beta) - 1) \overline{\varphi^\alpha_\beta} \wedge (\exp(\sigma \lambda_\alpha) - \exp(\sigma \lambda_\beta)) \varphi^\alpha_\beta \\
  &= \sum_{\alpha,\beta=1}^r (\exp(\lambda_\alpha - \lambda_\beta) - 1) (\exp(\sigma \lambda_\alpha) - \exp(\sigma \lambda_\beta)) |\varphi^\alpha_\beta|^2 \\
  &\geq \sum_{\alpha,\beta=1}^r \exp(- \sigma \lambda_\beta) (\exp(\sigma \lambda_\alpha) - \exp(\sigma \lambda_\beta))^2 |\varphi^\alpha_\beta|^2
\end{align*}
because $(\exp(x) - 1)(\exp(\sigma x) - 1) \geq (\exp(\sigma x) -
1)^2$
for all $x \in \RR$ and $0 \leq \sigma \leq 1$. The inequality
in \eqref{rs} now follows from
\begin{align*}
  \big|f^{-\frac{\sigma}{2}} [\varphi^\ast \, , f^\sigma]\big|^2
  &\,=\, - \tr_g \tr(f^{-\frac{\sigma}{2}} [\varphi^\ast \, , f^\sigma]
\wedge (f^{-\frac{\sigma}{2}} [\varphi^\ast \, , f^\sigma])^\ast) \\
  &=\, \sum_{\alpha,\beta=1}^r \exp(- \sigma \lambda_\beta) (\exp(\sigma
\lambda_\alpha) - \exp(\sigma \lambda_\beta))^2 |\varphi^\alpha_\beta|^2 \, .
  \qedhere
\end{align*}
\end{proof}

Now for $x \,\in\, M$, let $\lambda(\eps, x)$ be
the largest eigenvalue of $\log f_\eps(x)$. Define
\begin{equation}\label{17}
  M_\eps \, := \, \max_{x \in M} \lambda(\eps, x) \quad \text{and} \quad
  \rho_\eps \, := \, \exp(- M_\eps) \, .
\end{equation}
Since $\det f_\eps \,=\, 1$ by Lemma \ref{det}, it follows that
$\rho_\eps \,\leq\, 1$. As in \cite[Lemma 25]{Lo09}, we have the
following lemma.

\begin{lemma} \label{lemma3}
Assume that $\limsup_{\eps \to \infty} ||f_\eps||_{L^2}\,=\,
\infty$. Then
\begin{enumerate}
\item[(i)]   $\rho_\eps f_\eps \leq \id_E$, meaning that for
every $x \in M$, and every eigenvalue $\lambda$ of $\rho_\eps f_\eps(x)$, one has $\lambda \leq 1$,
\item[(ii)]  for every $x \in M$, there is an eigenvalue $c_x$ of
$\rho_\eps f_\eps(x)$ with $c_x\, \leq\, \rho_\eps$, where
$\rho_\eps$ and $f_\eps$ are defined in \eqref{17} and \eqref{14}
respectively,
\item[(iii)] $\max_M \rho_\eps |f_\eps| \leq 1$, and
\item[(iv)]  there is a sequence $\eps_i \,\longrightarrow\, 0$
such that $\rho_{\eps_i}\,\longrightarrow\, 0$.
\end{enumerate}
\end{lemma}

\begin{proposition} \label{weak-convergence}
There is a subsequence $\eps_i\,\longrightarrow\, 0$ such that
$\rho_{\eps_i}\,\longrightarrow\, 0$, and for $f_i :=\nolinebreak
\rho_{\eps_i} f_{\eps_i}$ (see Lemma \ref{lemma3} for notation),
\begin{enumerate}
\item[(i)]  $f_i$ converges weakly in $L_1^2$ norm to an $f_\infty \neq
0$, and
\item[(ii)] as $\sigma\,\longrightarrow\, 0$, $f_\infty^\sigma$
converges weakly in $L_1^2$ norm to an $f_\infty^0$.
\end{enumerate}
\end{proposition}

\begin{proof}
Apply the proof of \cite[Proposition 26]{Lo09} and use Corollary \ref{corollary1} (i), Proposition \ref{estimate-destabilizing} and Lemma \ref{lemma3}.
\end{proof}

Now let
\begin{equation}\label{pi}
\varpi := \id_E - f_\infty^0\, .
\end{equation}

\begin{proposition}
The endomorphism $\varpi$ in \eqref{pi} is an $h_0$--orthogonal
projection onto a
flat subbundle $F \,:=\, \varpi(E)$ of $E$ which is preserved by
the Higgs field $\varphi$, meaning it satisfies the identities
\[
  \varpi^2 \,=\, \varpi \, , \quad \varpi^\ast \,=\, \varpi \, ,
\quad (\id_E - \varpi) \circ \dbar \varpi \,=\, 0 \quad
\text{and} \quad
  (\id_E - \varpi) \circ \varphi \circ \varpi = 0 \quad \text{in
}
L^1 \, .
\]
Moreover, $\varpi$ is a smooth endomorphism of $E$. So the flat
subbundle $F$ is smooth.
\end{proposition}

\begin{proof}
Following the proof of \cite[Proposition 27]{Lo09}, using
Proposition
\ref{estimate-destabilizing}, Lemma \ref{lemma3} and
Proposition \ref{weak-convergence} we conclude that
\[
  \varpi^2 \,=\, \varpi \, , \quad \varpi^\ast \,=\, \varpi \quad
\text{and} \quad
(\id_E - \varpi) \circ \dbar \varpi \,=\, 0 \quad \text{in } L^1
\]
and that these imply that $\varpi$ is a smooth endomorphism of
$E$.

It remains to show that $(\id_E - \varpi) \circ \varphi \circ
\varpi \,=\,
0$, so that the smooth flat subbundle $F \,=\, \varpi(E)$ is
preserved by the Higgs field $\varphi$.

Applying the same argument as in \cite{Lo09} and using Proposition \ref{estimate-destabilizing}, we compute for $0 < \sigma \leq 1$ and $0 < s \leq \frac{\sigma}{2}$:
\begin{align*}
  \int_M \big|(\id_E - f_i^s) [\varphi^\ast \, , f_i^\sigma]\big|^2 \, \frac{\omega_g^n}{\nu}
  &\leq \left(\frac{s}{s + \frac{\sigma}{2}}\right)^2 \int_M \big|f_i^{-\frac{\sigma}{2}} [\varphi^\ast \, , f_i^\sigma]\big|^2 \, \frac{\omega_g^n}{\nu} \\
  &\leq \left(\frac{s}{s + \frac{\sigma}{2}}\right)^2 \int_M |\eps_i \log f_i + K_0^\varphi - \gamma \id_E| \cdot |f_i^\sigma| \, \frac{\omega_g^n}{\nu} \\
  &\leq \left(\frac{s}{s + \frac{\sigma}{2}}\right)^2 C \, .
\end{align*}
A similar argument to the one in \cite{Lo09} then gives that
\[
  \varpi \circ [\varphi^\ast \, , \id_E -\varpi] \,=\, 0 \, .
\]
Together with $\varpi^2 \,=\, \varpi$, this implies that
\[
  0 \,=\, - \varpi \circ [\varphi^\ast \, , \varpi]\,=\, \varpi \circ
\varphi^\ast \circ (\id_E - \varpi) \, ,
\]
and with $\varpi^\ast = \varpi$, it follows that
\[
  0 \,=\, \big(\varpi \circ \varphi^\ast \circ (\id_E -
\varpi)\big)^\ast \,=\, (\id_E - \varpi) \circ \varphi \circ \varpi
\, ,
\]
completing the proof of the proposition.
\end{proof}

\begin{proposition}
The flat subbundle $F\,=\, \varpi(E)\,\subset\, E$ is a proper
subbundle, meaning
\[
  0 \,<\, \rank F \,<\, \rank E \, .
\]
\end{proposition}

\begin{proof}
Apply the proof of \cite[Proposition 28]{Lo09}, and use Lemma
\ref{lemma3} and Proposition \ref{weak-convergence}.
\end{proof}

\begin{proposition}\label{32}
The flat subbundle $F \,=\, \varpi(E)$ is a destabilizing
subbundle,
meaning
\[
  \mu_g (F) \,\geq\, \mu_g (E) \, .
\]
\end{proposition}

\begin{proof}
As in \cite[proof of Proposition 29]{Lo09}, by the Chern--Weil formula we have
\[
  \mu_g (F) \,=\, \mu_g (E) + \frac{1}{sn} \int_M \big(\tr((K_0 -
\gamma
\id_E) \varpi) - |\del_0 \varpi|^2\big) \, \frac{\omega_g^n}{\nu}
\, ,
\]
where $s$ is the rank of $F$. Therefore, to complete the proof
it suffices to show that
\begin{equation}\label{as}
  \int_M \tr((K_0 - \gamma \id_E) \varpi) \,
\frac{\omega_g^n}{\nu}
\,\geq\, \int_M |\del_0 \varpi|^2 \, \frac{\omega_g^n}{\nu} \, .
\end{equation}

Using the identity $K_0^\varphi \,=\, K_0 + \tr_g [\varphi \, ,
\varphi^\ast]$ we obtain that
\begin{equation} \label{equation4}
  \int_M \tr((K_0 - \gamma \id_E) \varpi) \,
\frac{\omega_g^n}{\nu}
  = \int_M \tr((K_0^\varphi - \gamma \id_E) \varpi) \,
\frac{\omega_g^n}{\nu}
  - \int_M \tr(\tr_g [\varphi \, , \varphi^\ast] \varpi) \,
\frac{\omega_g^n}{\nu} \, .
\end{equation}
The first term in the right--hand side can be estimated as follows.
Since
\[
  \varpi \,=\, \lim_{\sigma \to 0}\limits \, \lim_{i \to \infty}\limits \, (\id_E - f_i^\sigma)
\]
strongly in $L^2$ norm, and $\tr(K_0^\varphi - \gamma \id_E) \,=\, 0$, we have
\[
  \int_M \tr((K_0^\varphi - \gamma \id_E) \varpi) \,
\frac{\omega_g^n}{\nu}
\,= \, - \lim_{\sigma \to 0} \, \lim_{i \to \infty} \int_M
\tr((K_0^\varphi - \gamma \id_E) f_i^\sigma) \, \frac{\omega_g^n}{\nu}\, ,
\]
and using equation \eqref{equation1}, we see that
\[ \begin{split}
  - \int_M \tr((K_0^\varphi - \gamma \id_E) f_i^\sigma) \, \frac{\omega_g^n}{\nu}
  = &\int_M \eps_i \tr(\log(f_{\eps_i}) f_i^\sigma) \, \frac{\omega_g^n}{\nu} \\
    &+ \int_M \tr(\tr_g \dbar(f_i^{-1} \del_0 f_i) f_i^\sigma) \, \frac{\omega_g^n}{\nu} \\
    &+ \int_M \tr(\tr_g [\varphi \, , f_i^{-1} [\varphi^\ast \, , f_i]] f_i^\sigma) \, \frac{\omega_g^n}{\nu} \, .
\end{split} \]
We estimate the first two integrals as in \cite{Lo09} and the
third integral as in the proof of Proposition
\ref{estimate-destabilizing}. Together with $f_i \leq \id_E$, it
then follows that
$$
  - \int_M \tr((K_0^\varphi - \gamma \id_E) f_i^\sigma) \, \frac{\omega_g^n}{\nu}
\,\geq\, ||\del_0 (\id_E - f_i^\sigma)||_{L^2}^2 +
\big\lVert[\varphi^\ast \, , \id_E - f_i^\sigma]\big\rVert_{L^2}^2 \, .
$$
Passing to the limit $i\longrightarrow\infty$ as
in \cite{Lo09}, we obtain the following estimate of \eqref{equation4}:
\begin{equation}\label{as2}
  \int_M \tr((K_0 - \gamma \id_E) \varpi) \,
\frac{\omega_g^n}{\nu}
  \,\geq\, ||\del_0 \varpi||_{L^2}^2 + \big\lVert [\varphi^\ast
\, ,
\varpi] \big\rVert_{L^2}^2
  - \int_M \tr(\tr_g [\varphi \, , \varphi^\ast] \varpi) \,
\frac{\omega_g^n}{\nu} \, .
\end{equation}

Now, using $\varpi^2 \,=\, \varpi$, $\varpi^\ast \,=\,
\varpi$ and $(\id_E -
\varpi) \circ \varphi \circ \varpi \,=\, 0$, one shows that
\[
  \big\lVert [\varphi^\ast \, , \varpi] \big\rVert_{L^2}^2
 \, =\,\int_M \tr(\tr_g [\varphi \, , \varphi^\ast] \varpi) \,
\frac{\omega_g^n}{\nu} \, .
\]
Therefore, the inequality in \eqref{as} follows from
the one in \eqref{as2}. This completes the proof of the
proposition.
\end{proof}

Proposition \ref{32} completes the proof of Theorem \ref{main}.

\section{Some consequences}

Theorem \ref{main} has the following corollary:

\begin{corollary} \label{polystable}
Let $M$ be a compact special affine manifold equipped with an
affine Gauduchon metric $g$, and let $(E \, , \varphi)$ be a
flat Higgs vector bundle over $M$. Then $E$ admits an affine
Yang--Mills--Higgs metric if and only if it is polystable.
Moreover, a polystable flat Higgs vector bundle over $M$ admits
a unique Yang--Mills--Higgs connection.
\end{corollary}

\begin{proof}
The ``if'' part follows immediately from Theorem \ref{main}.

For the ``only if'' part, assume that $(E \, , \varphi)$ admits
an affine Yang--Mills--Higgs metric $h$. Let $F$ be a flat
subbundle of $E$ which is preserved by the Higgs field
$\varphi$.

The flat connection on $E$ will be denoted by $\nabla$. Let
$\nabla_F$ be the flat connection on $F$ induced by
$\nabla$, and let $h_F$ be the Hermitian metric on $F$ induced by
$h$. Then for any section $s$ of $F$, we have
\[
  \del^{\nabla, h} s \,=\, \del^{\nabla_F, h_F} s + A(s) \, ,
\]
where $A \,\in\, \A^{1,0}(\Hom(F \, , F^\perp))$ is the second
fundamental form, and $\del^{\nabla, h}$ (respectively,
$\del^{\nabla_F, h_F}$) is the component of type $(1\, ,0)$ of the
extended Hermitian connection on $E$ (respectively, $F$) with respect to
$h$ (respectively, $h_F$). Analogously, if
$\varphi_F$ is the flat Higgs
field on $F$ induced by $\varphi$, we write
\[
  \varphi^\ast(s) = \varphi_F^\ast(s) + \widetilde \varphi(s) \, ,
\]
where $\varphi^\ast$ and $\varphi_F^\ast$ are the adjoints with
respect to $h$ and $h_F$, respectively, and $\widetilde \varphi$
is a $(0 \, , 1)$ form with values in $\Hom(F \, , F^\perp)$.

To complete the proof of the ``only if'' part, it suffices to show
that
$\mu_g (F) \,\leq\, \mu_g (E)$ with the equality holding if
and only if $A$ and $\widetilde \varphi$ vanish identically.

Denoting by $s$ the rank of $F$, we compute
\[
\mu_g (F) \,=\, \mu_g (E) - \frac{1}{sn} \int_M |A|^2 \,
\frac{\omega_g^n}{\nu} - \frac{1}{sn} \int_M |\widetilde \varphi|^2 \, \frac{\omega_g^n}{\nu} \, ,
\]
which implies that $\mu_g (F) \,\leq\, \mu_g (E)$ with the equality
holding if
and only if $A$ and $\widetilde \varphi$ vanish identically.

To prove the uniqueness of the Yang--Mills--Higgs connection,
first note that a stable flat Higgs bundle on $M$ admits a
unique Yang--Mills--Higgs connection, because any
two Yang--Mills--Higgs metrics on it differ by a constant
scalar (see Proposition \ref{uniqueness} and Proposition
\ref{stable-simple}). Write a polystable flat Higgs bundle
$(E \, , \varphi)$ as a direct sum of stable flat Higgs bundles.
It was shown above that a
Yang--Mills--Higgs connection on $(E \, , \varphi)$
is the direct sum of Yang--Mills--Higgs connections on the
stable direct summands. Therefore, $(E \, , \varphi)$ admits
a unique Yang--Mills--Higgs connection.
\end{proof}

Let us observe that the above results also hold for flat
real Higgs bundles.

\begin{definition} \label{real1}
Let $(E \, , \varphi)$ be a flat real Higgs bundle on a compact special affine manifold $M$ equipped with an affine Gauduchon metric $g$.
\begin{enumerate}
\item[(i)] $(E \, , \varphi)$ is called {\em $\RR$--stable\/}
(respectively, {\em $\RR$--semistable}) if for every flat real subbundle
$F$ of $E$, with $0 < \rank F < \rank E$, which is preserved by
$\varphi$, we have
\[
  \mu_g (F) \,<\, \mu_g (E) \quad
  \text{(respectively, }
  \mu_g (F)\, \leq\, \mu_g (E)
  \text{)} \, .
\]
\item[(ii)] $(E \, , \varphi)$ is called {\em $\RR$--polystable\/} if
\[
(E \, , \varphi)\, =\, \bigoplus_{i=1}^N \, (E_i \, , \varphi_i)\, ,
\]
where each $(E_i \, , \varphi_i)$ is an $\RR$--stable flat real Higgs
bundle with $\mu_g (E_i)\, = \,\mu_g (E)$.
\end{enumerate}
\end{definition}

\begin{corollary} \label{real2}
Let $M$ be a compact special affine manifold
equipped with an affine Gauduchon metric, and let $(E \, ,
\varphi)$ be a flat real Higgs vector bundle over $M$. Then $(E
\, , \varphi)$ admits an affine Yang--Mills--Higgs metric if and
only if it is $\RR$--polystable. Moreover, a polystable flat
real Higgs vector bundle over $M$ admits
a unique Yang--Mills--Higgs connection.
\end{corollary}

\begin{proof}
This follows from Corollary \ref{polystable} as
in \cite[Section 11]{Lo09}.
\end{proof}

\subsection{Flat Higgs $G$--bundles}\label{se-G}

Any flat (real or complex) vector bundle over a compact affine
manifold equipped with an affine Gauduchon metric has a
unique Harder--Narasimhan filtration \cite{BL}. Using it and
the above mentioned correspondence in \cite{Lo09},
the following can be proved:

\begin{theorem}[\cite{BL}]\label{tBL}
Let $G$ be a reductive complex affine algebraic group.
Let $M$ be a compact special affine manifold equipped with an
affine Gauduchon metric, and let $E_G$ be a
flat principal $G$--bundle over $M$. Then $E$ admits an affine
Yang--Mills connection if and only if $E_G$ is polystable.
Further, the Yang--Mills connection on a polystable flat
bundle is unique.
\end{theorem}

The above result remains valid if $G$ is a reductive
affine algebraic group over $\mathbb R$ of split type \cite{BL}.

The proof of the existence and uniqueness of the
Harder--Narasimhan filtration of a flat vector bundle goes
through for a flat (real or complex) Higgs vector bundle.
So a (real or complex) flat Higgs vector bundle over a
compact affine manifold equipped with an affine Gauduchon metric
has a unique Harder--Narasimhan filtration.

Let $G$ be a reductive algebraic group. Let $M$ be
a compact affine manifold equipped with an
affine Gauduchon metric $g$.
Let $E_G$ be a principal $G$--bundle over $M$
equipped with a flat connection $\nabla^G$. Let
$$
\text{ad}(E_G)\, :=\, E_G\times^G \text{Lie}(G)
$$
be the adjoint vector bundle over $M$ associated to $E_G$.
Since the adjoint action of $G$ on $\text{Lie}(G)$ preserves
the Lie algebra structure, each fiber of $\text{ad}(E_G)$
is a Lie algebra isomorphic to $\text{Lie}(G)$.
If $\varphi$ is a smooth section of $T^\ast M \tensor
\text{ad}(E_G)$, then using the Lie algebra structure of the
fibers of $\text{ad}(E_G)$, and the obvious projection
$T^\ast M\otimes  T^\ast M\, \longrightarrow\, \bigwedge^2
T^\ast M$, we get a smooth section of
$(\bigwedge^2 T^\ast M) \tensor \text{ad}(E_G)$, which we will
denote by $[\varphi \, , \varphi]$.

The flat connection $\nabla^G$ on $E_G$ induces a flat connection
on
$\text{ad}(E_G)$; this flat connection on $\text{ad}(E_G)$
will be denoted by $\nabla^{\rm ad}$. Let
$$
\widetilde{\nabla}^{\rm ad}\, :\,  T^\ast M \tensor
\text{ad}(E_G)\,
\longrightarrow\, T^\ast M \tensor T^\ast M \tensor\text{ad}(E_G)
$$
be the flat connection on $T^\ast M \tensor \text{ad}(E_G)$
defined by $\nabla^{\rm ad}$ and the connection $D^*$ on
$T^\ast M$.

A \textit{Higgs field\/} on the flat principal $G$--bundle $(E_G\,
,\nabla^G)$ is a smooth section $\varphi$ of $T^\ast M \tensor
\text{ad}(E_G)$ such that
\begin{enumerate}
\item the section $\varphi$ is flat with respect to the
connection $\widetilde{\nabla}^{\rm ad}$ on
$T^\ast M \tensor \text{ad}(E_G)$, and

\item $[\varphi \, , \varphi] \,=\, 0$.
\end{enumerate}
A \textit{Higgs $G$--bundle\/} is a flat principal $G$--bundle
together with a Higgs field on it. (See \cite{Si2} for Higgs
$G$--bundles on complex manifolds.)

Let $(E_G\, ,\nabla^G\, ,\varphi)$ be a Higgs $G$--bundle on $M$.
Fix a maximal compact subgroup $K\, \subset\, G$. Given a
$C^\infty$ reduction of structure group $E_K\, \subset\, E_G$,
we have a natural connection $\nabla^{E_K}$ on the principal
$K$--bundle $E_K$ constructed using $\nabla^G$; the
connection on $E_G$ induced by $\nabla^{E_K}$ will also be
denoted by $\nabla^{E_K}$. Given a
$C^\infty$ reduction of structure group $E_K\, \subset\, E_G$
to $K$, we may define as before the $(1\, ,1)$--part of the
extended curvature
$$
\overline{\partial} \theta + [\varphi \, , \varphi^\ast]\, ,
$$
which is a $(1\, ,1)$--form with values in $\text{ad}(E_G)$;
as before, $\theta$ is a $(1\, ,0)$--form with values in
$\text{ad}(E_G)$.

The reduction $E_K$ is called a \textit{Yang--Mills--Higgs\/}
reduction of $(E_G\, ,\nabla^G\, ,\varphi)$ if
there is an element $\gamma$ of the center of
$\text{Lie}(G)$ such that the section
$$
\tr_g(\overline{\partial} \theta + [\varphi \, , \varphi^\ast])
$$
of $\text{ad}(E_G)$ coincides with the
one given by $\gamma$. If $E_K$ is a Yang--Mills--Higgs
reduction, then the connection $\nabla^{E_K}$ on $E_G$ is
called a \textit{Yang--Mills--Higgs connection}.

The proof of Theorem \ref{tBL} (see \cite{BL}) gives the following:

\begin{corollary}\label{cor-G}
Let $M$ be a compact special affine manifold
equipped with an affine Gauduchon metric.
Let $G$ be either a reductive affine algebraic group over
$\mathbb C$ or a reductive
affine algebraic group over $\mathbb R$ of split type.
Then a flat Higgs $G$--bundle $(E_G \, , \varphi)$ over $M$
admits a Yang--Mills--Higgs connection if and only
if $(E_G \, , \varphi)$ is polystable.
Further, the Yang--Mills--Higgs connection on a polystable flat
Higgs $G$--bundle is unique.
\end{corollary}

\subsection{A Bogomolov inequality}

As before, $M$ is a compact special affine manifold
of dimension $n$ equipped with a
Gauduchon metric $g$. We assume that $g$ is {\em astheno--K\"ahler}, meaning
\begin{equation}\label{aK}
\partial\overline{\partial}(\omega^{n-2}_g)\, =\,0\, ,
\end{equation}
where $\omega_g$ is defined in \eqref{deomg}
(see \cite[p.\ 246]{JY}).

\begin{proposition}\label{prop-Bo}
Let $(E\, ,\varphi)$ be a semistable flat Higgs vector bundle
of rank $r$ over $M$. Then
$$
\int_M \frac{c_2(\End(E)) \wedge \omega^{n-2}_g}{\nu}\, =\,
\int_M \frac{(2r\cdot c_2(E) -
(r-1)c_1(E)^2) \wedge \omega^{n-2}_g}{\nu}\, \geq\, 0\, .
$$
\end{proposition}

\begin{proof}
First assume that $(E\, ,\varphi)$ is a polystable
flat Higgs vector bundle. Consider an affine Yang--Mills--Higgs
metric $h$ on $E$ given by Theorem \ref{main}. Then the integral of the $n$--form
$$
\frac{(2r\cdot c_2(E \,, h) -
(r-1)c_1(E \,, h)^2) \wedge \omega^{n-2}_g}{\nu}
$$
on $M$ coincides with the integral of a pointwise nonnegative $n$--form
(see \cite[p.\ 878--879, Proposition 3.4]{Si88}
and also \cite[p.\ 107]{LYZ} for the computation); here $\nu$
is the covariant constant volume form. Therefore,
$$
\int_M \frac{(2r\cdot c_2(E \,, h) -
(r-1)c_1(E \,, h)^2) \wedge \omega^{n-2}_g}{\nu}\, \geq\, 0\, .
$$
Hence the inequality in the proposition is proved for polystable
Higgs vector bundles.

If the flat Higgs bundle
$(E\, ,\varphi)$ is semistable, then there is a filtration of flat subbundles
$$
0\,=\, E_0\, \subset\, E_1\, \subset\, \cdots
\, \subset\, E_{\ell-1} \, \subset\, E_\ell\, =\, E
$$
such that
\begin{itemize}
\item $\varphi(E_i)\, \subset\, T^*M\otimes E_i\, \subset\, T^*M\otimes E$
for all $i\, \in\, [0 \,, \ell]$,

\item the quotient $E_i/E_{i-1}$ equipped with the Higgs field induced
by $\varphi$ is polystable for each $i\, \in\, [1 \,, \ell]$, and

\item $\mu_g(E_i/E_{i-1})\,=\, \mu_g(E)$ for each $i\, \in\, [1 \,, \ell]$.
\end{itemize}

We have shown that the inequality in the proposition holds for each
$E_i/E_{i-1}$, $i\, \in\, [1 \,, \ell]$. Therefore, the inequality
holds for $E$.
\end{proof}

For a semistable flat Higgs $G$--bundle $(E_G \, , \varphi)$ over $M$,
the adjoint vector bundle $\text{ad}(E_G)$ equipped with the Higgs
field induced by $\varphi$ is also semistable. Therefore, Proposition
\ref{prop-Bo} gives a similar inequality for semistable flat Higgs
$G$--bundles.


\begin{thebibliography}{Do87b}

\bibitem[BL11]{BL} {\scshape I.~Biswas and J.~Loftin}:
Hermitian--Einstein connections on principal bundles over flat
affine manifolds, \textit{Int.\ Jour.\ Math.\/}\ \textbf{23},
no.\ 4 (2012), {\ttfamily doi:10.1142/S0129167X12500395}.

\bibitem[Do85]{Do85} {\scshape S.~K.~Donaldson}:
Anti self--dual Yang--Mills connections over
complex algebraic surfaces and stable vector bundles,
\textit{Proc.\ London Math.\ Soc.\/}\ \textbf{50} (1985), 1--26.

\bibitem[Do87a]{Do87} {\scshape S.~K.~Donaldson}: Infinite
determinants, stable bundles and curvature, \textit{Duke Math.\
Jour.\/}\ \textbf{54} (1987), 231--247.

\bibitem[Do87b]{Do} {\scshape S.~K.~Donaldson}: Twisted harmonic
maps and self--duality equations, \textit{Proc.\ London Math.\ Soc.\/}\ \textbf{55} (1987), 127--131.

\bibitem[Hi87]{Hi} {\scshape N.~J.~Hitchin}: The
self--duality equations on a Riemann surface,
\textit{Proc.\ London Math.\ Soc.\/}\ \textbf{55} (1987), 59--126.

\bibitem[Ja11]{Ja11} {\scshape A.~Jacob}: Stable Higgs
bundles and Hermitian--Einstein metrics on non--K\"ahler
manifolds, {\ttfamily arXiv:1110.3768v1} [math.DG] (2011).

\bibitem[JY93]{JY} {\scshape J.~Jost and S.--T.~Yau}: A nonlinear elliptic
system for maps from Hermitian to Riemannian manifolds and
rigidity theorems in Hermitian geometry, \textit{Acta Math.\/}\
\textbf{170} (1993), 221--254.

\bibitem[LY87]{LY87} {\scshape J.~Li and S.--T.~Yau}:
Hermitian--Yang--Mills connection on non--K\"ahler manifolds,
\textit{Mathematical aspects of string theory}, Proc.\ Conf., San
Diego/Calif.\ 1986, 560--573, Adv.\ Ser.\ Math.\ Phys.\
{\bfseries 1} (1987).

\bibitem[LYZ]{LYZ} {\scshape J.~Li, S.--T.~Yau and F.~Zheng}: On
projectively flat Hermitian manifolds, \textit{Comm.\ Anal.\ Geom.\/}
\ \textbf{2} (1994), 103--109.

\bibitem[Lo09]{Lo09} {\scshape J.~Loftin}: Affine
Hermitian--Einstein metrics, \textit{Asian J.\ Math.\/}\ {\bfseries
13} (2009), 101--130.

\bibitem[LT95]{LT95} {\scshape M.~L\"ubke and A.~Teleman}: {\em
The Kobayashi--Hitchin correspondence}, Singapore: World
Scientific (1995).

\bibitem[Sc05]{Sc1} {\scshape L.~Sch\"afer}: $tt^\ast$--geometry and pluriharmonic
maps, \textit{Ann.\ Glob.\ Anal.\ Geom.\/}\ \textbf{28} (2005), 285--300.

\bibitem[Sc07]{Sc2} {\scshape L.~Sch\"afer}: A note on $tt^\ast$--bundles over
compact nearly K\"ahler manifolds, \textit{Geom.\ Dedicata\/} \textbf{128}
(2007), 107--112.

\bibitem[Si88]{Si88} {\scshape C.~T.~Simpson}: Constructing
variations of Hodge structure using Yang--Mills theory and
applications to uniformization, \textit{Jour.\ Amer.\ Math.\
Soc.\/}\ {\bfseries 1} (1988), 867--918.

\bibitem[Si92]{Si2} {\scshape C.~T.~Simpson}: Higgs bundles and
local systems, \textit{Inst.\ Hautes \'Etudes Sci.\ Publ.\ Math.\/}\ \textbf{75}
(1992), 5--95.

\bibitem[UY86]{UY86} {\scshape K.~Uhlenbeck and S.--T.~Yau}: On
the existence of Hermitian--Yang--Mills connections in stable
vector bundles, \textit{Commun.\ Pure Appl.\ Math.\/}\ \textbf{39}
(1986), 257--293.

\bibitem[UY89]{UY89} {\scshape K.~Uhlenbeck and S.--T.~Yau}: A
note on our previous paper: On the existence of Hermitian
Yang--Mills connections in stable vector bundles, \textit{Commun.\
Pure Appl.\ Math.\/}\ \textbf{42} (1989), 703--707.

\end{thebibliography}
\end{document}